\documentclass[11pt,a4paper]{article}
\usepackage[utf8x]{inputenc}
\usepackage{ucs}
\usepackage{amsmath}
\usepackage{amsfonts}
\usepackage{amssymb}
\usepackage{amsthm}
\usepackage{graphicx}
\usepackage{color}
\usepackage{geometry}
\usepackage{tikz}
\newtheorem{lem}{Lemma}
\newtheorem{rem}{Remark}
\newtheorem{thm}{Theorem}
\newtheorem{prop}{Proposition}
\newtheorem{cor}{Corollary}
\newcommand{\E}{\mathbb{E}}
\title{Universality of the nodal length of bivariate random trigonometric polynomials}
\author{J\"urgen Angst, Guillaume Poly, Hung Pham Viet}

\begin{document}

\maketitle

\begin{abstract}We consider random trigonometric polynomials of the form 
\[
f_n(x,y)=\sum_{1\le k,l \le n} a_{k,l} \cos(kx) \cos(ly),
\]
where the entries $(a_{k,l})_{k,l\ge 1}$ are i.i.d. random variables that are centered with unit variance. We investigate the length $\ell_K(f_n)$ of the nodal set $Z_K(f_n)$ of the zeros of $f_n$ that belong to a compact set $K \subset \mathbb R^2$. We first establish a local universality result, namely we prove that, as $n$ goes to infinity, the sequence of random variables $n\,  \ell_{K/n}(f_n)$ converges in distribution to a universal limit which does not depend on the particular law of the entries. We then show that at a macroscopic scale, 
the expectation of $\ell_{[0,\pi]^2}(f_n)/n$ also converges to an universal limit. Our approach provides two main byproducts: (i) a general result regarding the continuity of the volume of the nodal sets with respect to $C^1$-convergence which refines previous findings of \cite{rusakov,iksanov2016local,azais2015local} and (ii) a new strategy for proving small ball estimates in random trigonometric models, providing in turn uniform local controls of the nodal volumes.
\end{abstract}

\setcounter{secnumdepth}{3} 
\setcounter{tocdepth}{3}
\tableofcontents

\section{Introduction}

The study of nodal sets associated to various kinds of random functions is a central topic of probability theory, at the crossroad of various domains of mathematics and physics such as linear algebra, number theory, geometric measure theory or else quantum mechanics or nuclear physics, just to name a few. In this context, universality results refer to asymptotic properties of these random nodal domains, which hold regardless of the nature of the randomness involved. 
Establishing such universal properties for generic zero sets allows one to manage what would be otherwise inextricable objects, which explains the tremendous importance of this particular area of research. As such, the literature on this topic is huge and we refer to  the introduction of \cite{tao2014local} and the references therein for a general overview.

\medskip

When the random functions under consideration are multivariate, the zeros  are no more isolated points but instead random curves/surfaces/manifolds whose volume is, among others, a natural quantity of interest. Ranging from algebraic manifolds to nodal lines of random eigenfunctions of Laplace-Beltrami operators on tori or spheres, this topic has attracted a lot of attention very recently. Non exhaustively, we refer for instance to \cite{zelditch,gayet2012betti,letendre2016expected,letendre2016variance} regarding random algebraic manifolds and to \cite{rudnick2008volume,oravecz2008leray,wigman2010fluctuations,nazarov2010random,fyodorov2015number,marinucci2015non} regarding random eigenfunctions. Nevertheless, in each situation considered in the above references, the underlying randomness emerges from Gaussian distribution and there seems to be actually no results dealing with the dependency of the studied phenomena on the particular choice of the distribution of the randomness. One reason possibly explaining the lack of results of universality in multivariate frameworks is that most techniques successfully used in univariate settings, such as complex analysis tools or else counting the changes of sign, seem hardly extendable to higher dimensions. For instance, to the best of our knowledge, there is no simple analogous in $\mathbb{C}^2$ of the Jensen formula which plays a central role in universality questions for univariate algebraic polynomials, see \cite{tao2014local}. On the opposite, we point out the fact that whatever the dimension is, a Kac--Rice formula still holds and allows one to manage remarkably well the case of absolutely continuous random fields. In this article, we investigate the natural question of asymptotic universality of volumes in the framework of bivariate random trigonometric polynomials with random coefficients that are only assumed to be i.i.d and standardized. Let us describe below our model in details.

\medskip

Let $(a_{k,l})_{k,l\ge 1}$ be a sequence of independent and identically distributed random variables whose common law is centered with unit variance. We consider the random function $f_n : \mathbb R^2 \to \mathbb R$ defined as
\begin{equation} \label{def.poly}
f_n(x,y)=\sum_{1\le k,l \le n} a_{k,l} \cos(kx) \cos(ly), \;\; (x,y) \in \mathbb R^2,
\end{equation}
and its renormalized analogue
\begin{equation} \label{def.poly2}
F_n(x,y) :=\frac{1}{n} f_n\left(\frac{x}{n}, \frac{y}{n}\right)= \frac{1}{n} \sum_{1\leq k, \ell, \leq n} a_{k , \ell} \cos \left( \frac{k x}{n}\right)\cos \left( \frac{\ell y }{n}\right).
\end{equation}
We denote by $Z_K(f)$ the zeros of a function $f$ in a compact set $K \subset \mathbb R^2$ and by $\ell_K(f)$ the length, or $1-$dimensional Hausdorff measure of $Z_K(f)$ : 
\[
\ell_K(f) := | Z_K(f)|, \;\; \hbox{where} \;\; Z_K(f):=\{(x,y) \in K \subset \mathbb R^2, \; f(x,y)=0\}.
\]
\begin{figure}[ht]
\begin{center}
\includegraphics[scale=0.25]{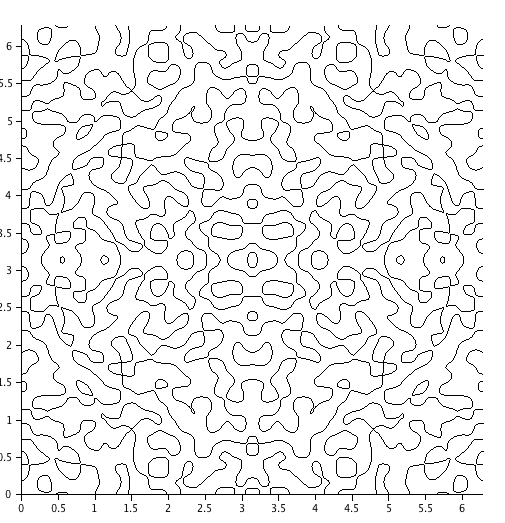}\includegraphics[scale=0.25]{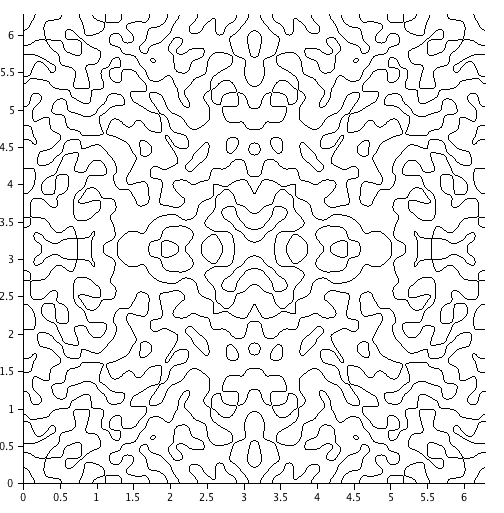}\includegraphics[scale=0.3]{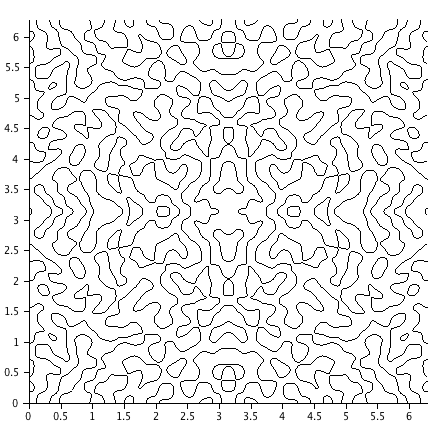}
\end{center}
\caption{A realization of the random nodal set $Z_K(f_n)$ for $K=[0,2\pi]^2$, $n=20$, with, from left to right, Bernoulli, Gaussian, centered exponential entries.}
\end{figure}\par
\noindent

\medskip

Our first main result is the following local universality result which states that, at a microscopic scale, the length of the nodal set converges in distribution to a universal limit. 

\begin{thm}[Local universality, Theorem \ref{theo.local} below]\label{theo.main1}
For any fixed compact $K \subset \mathbb R^2$, the sequence of random variables $(\ell_K(F_n))_{n \geq 1}$ converges in distribution, as $n$ tends to infinity, to an explicit random variable, whose law is independent of the particular law of the entries $(a_{k,l})_{k,l\ge 1}$.
\end{thm}

In comparison with the recent work \cite{iksanov2016local} which rather uses complex analysis and Hurwitz Theorem, we actually show that the sole $C^1$-convergence is enough to ensure local universality.  Besides, even if stated here in dimension two, our result holds in any finite dimension. Nevertheless, in \cite{iksanov2016local}, a much wider class of distributions is considered, englobing domains of attraction of stable distributions. The article \cite{azais2015local} provides local universality for some families of absolutely continuous distributions which is an unecessary assumption but actually entails the stronger result that all moments converge towards the corresponding moments of the (moment determined) target.

\medskip

From the above local universality result, and provided explicit moment controls, we can then deduce the following global universality result, which states that, properly nomalized, the expectation of the length of the full nodal set in the square $[0,\pi]^2$, converges to a universal constant.

\begin{thm}[Global universality, Theorem \ref{theo.final} below]\label{theo.main}
Whatever the law of the entries $(a_{k,l})_{k,l\ge 1}$,  as $n$ tends to infinity, we have 
\[
\lim_{n \to +\infty} \frac{\mathbb E[ \ell_{[0,\pi]^2}(f_n) ]}{n} = \frac{\pi^2}{2 \sqrt{3}}.
\]
\end{thm}

\begin{rem}
Due to the symmetry and periodicity of the trigonometric polynomials $f_n$, we have then 
$\lim_{n \to +\infty} n^{-1} \mathbb E[ \ell_{[0,2\pi]^2}(f_n) ]= 2\pi^2/\sqrt{3}$,
and our proof actually establishes that for any compact set $K$ being a finite union of rectangles:
\[
\lim_{n \to +\infty} \frac{\mathbb E[ \ell_{K}(f_n) ]}{n} = \frac{\text{Vol}(K)}{2 \sqrt{3}}.
\]
With a standard approximation procedure, one can then deduce that the latter convergence holds for any compact set $K$ with non empty interior and smooth boundary. 
\end{rem}

\begin{figure}[ht]
\begin{center}\includegraphics[scale=0.4]{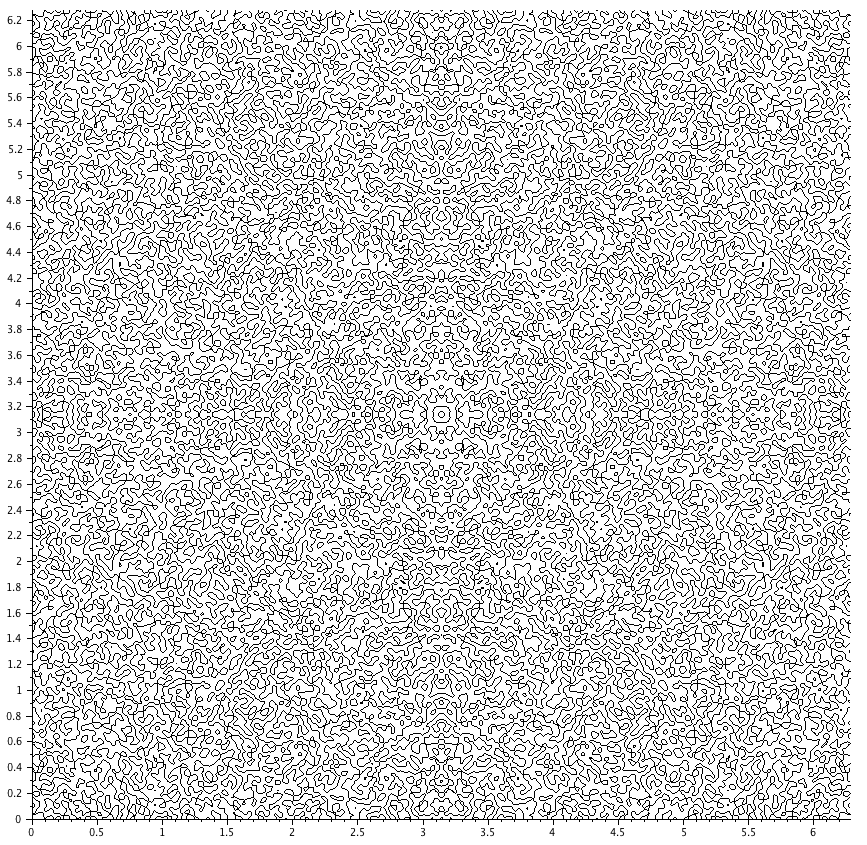}\end{center}
\caption{A realization of the nodal set $Z_{[0,2\pi]^2}(f_n)$ for a trigonometric polynomial of degree $n=100$ and with symmetric Bernoulli coefficients.}
\end{figure}

\begin{rem}\label{rem.nonstat}
By choosing the trigonometric polynomials $f_n$ of the form given by Equation \eqref{def.poly}, we deliberately choose to work in a non-stationnary framework. Let us stress here that our methods and results naturally extend to stationnary cases, for instance when the trigonometric polynomials are of the form
\[
\sum_{1\le k,l \le n} a_{k,l} \cos(kx+ly) + b_{k,l} \sin( kx+ly), 
\]
where $a_{k,l}$ and $b_{k,l}$ are independant i.i.d sequences and where the computations are actually simpler than the ones considered here. 
\end{rem}

Before giving the plan of the paper, let us say few words concerning the universality of the mean number of real roots of univariate random trigonometric polynomials. It has been recently established in full generality in \cite{flasche2016expected}, and in \cite{angst2015universality} under more restrictive conditions on the coefficients but with some possible control of the remainder in terms of Edgeworth expansions. The strategy of the proof in \cite{flasche2016expected} artfully combines a careful investigation of the number of changes of signs together with accurate small ball estimates obtained by adapting to this framework the method of Ibragimov and Maslova \cite{ibragimov1971expected}. Nevertheless such a strategy faces intricate obstructions in higher dimensions, first of all, investigating the number of changes of sign is not anymore suitable. Secondly, relying on the celebrated Crofton formula, one might try to get back to the univariate case by studying only the zeros of our bivariate polynomials when restricted to random lines. However, such projections are not anymore polynomials when the lines have an irrational slope. In order to avoid such heavy complications, we follow here a completely different path which consists of establishing first the local universality and next extending it to global universality via accurate controls of moments of of local nodal lengths. These controls rely on suitable small ball estimates which do not follow the Ibragimov-Maslova method, which seemed to us hard to adapt here, but instead  exploit the particular ergodic properties of sequences of type $\{k x\}_{k \ge 1} \,\,\text{mod}(\pi)$.

\medskip

The plan of the paper is the following. The next Section \ref{Loc} is devoted to the proof of Theorem \ref{theo.main1} concerning local universality. 
Its first Subsection \ref{sec.C1} is dedicated to the $C^1$-convergence of the rescaled trigonometric polynomials $F_n$ towards a non-degenerate Gaussian field, whereas Subsection \ref{sec.continuity} deals with the (deterministic) continuity of the volumes of nodal domains with respect to $C^1$-convergence on compact sets. The last two results are combined in Subsection \ref{Locuniv} to deduce the announced microscopic universality.
The proof of Theorem \ref{theo.main} on global universality is then given in Section \ref{sec.global}. More precisely, Subsection \ref{sec.gaussian} deals with the Gaussian case, where an exact computation of the nodal length can be performed thanks to the celebrated Kac--Rice formula.   
Then, in Subsection \ref{sec.moment}, we derive a  small ball estimate, from which we deduce a uniform moment control of the local lengths. Together with the local universality, this moment control allow us to conclude in Subsection \ref{lafin}.
For the sake of clarity, we give below a concise view of our proof strategy.

\begin{figure}[ht]
\begin{center}
\includegraphics[scale=0.45]{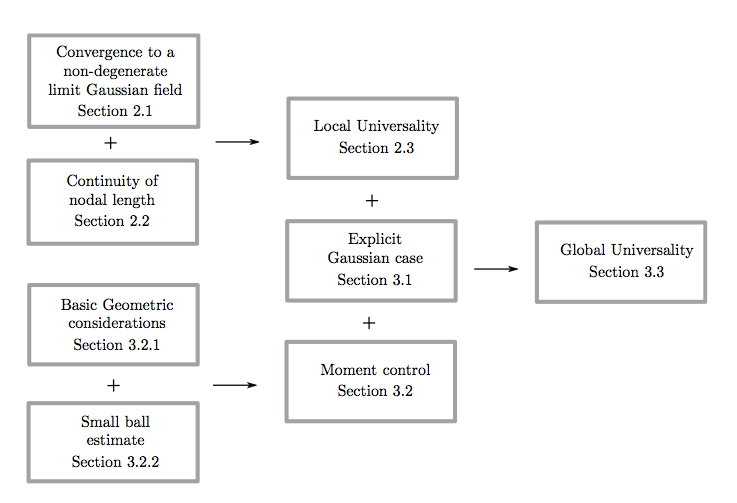}
\end{center}
\vspace{-1cm}
\caption{Plan of the proof of Local/Global Universality.}
\end{figure}

\section{Local universality} \label{Loc}
 In this section, we give a detailed proof of Theorem \ref{theo.main1} on the local universality of the nodal length, i.e. we show that, at the microscopic scale, the law of the nodal length of the bivariate random trigonometric polynomials converges to a universal limit as their degree tends to infinity, regardless of the particular law of their coefficients.

\subsection{A limit Gaussian field} \label{sec.C1}
Let us first remark that, up to a scale factor, the set of zeros of the original random trigonometric polynomial $f_n$ defined by Equation \eqref{def.poly} naturally identifies with the set of zeros of its rescaled analogue $F_n$ defined by Equation \eqref{def.poly2}.  
But the advantage of considering the function $F_n$ instead of $f_n$ is that, for any fixed compact $K \subset \mathbb R^2$ and as $n$ goes to infinity, the random field $(F_n(x,y))_{(x,y) \in K}$ converges in law, with respect to the $C^1$ topology, to an explicit smooth Gaussian field $(F_{\infty}(x,y))_{(x,y) \in K}$.

\begin{prop} \label{conv.C1}
For any fixed compact $K \subset \mathbb R^2$, as $n$ goes to infinity, the random field $(F_n(x,y))_{(x,y) \in K}$ converges with respect to the $C^1$ topology on $K$ to a smooth, centered Gaussian field $(F_{\infty}(x,y))_{(x,y) \in K}$ whose covariation is given by
\[
\begin{array}{ll}
\mathbb E[F_{\infty}(x,y)F_{\infty}(x',y')] & = \displaystyle{\int_0^1 \int_0^1 \cos( s x) \cos(s x') \cos( t y) \cos(t y') ds dt }\\
\\
& = \displaystyle{ \frac{1}{4} \left( \sin_c(x+x') +  \sin_c(x-x') \right)\left( \sin_c(y+y') +  \sin_c(y-y') \right)},
\end{array}
\]
where $\sin_c(x) := \sin(x)/x$ if $x \neq 0$ and $\sin_c(0):=1$ by convention.
\end{prop}

\begin{proof}
We use here the characterization of the $C^1$-convergence given in Theorem 2 and Remarks 2 and 3 of \cite{rusakov}. The convergence of finite dimensional marginals is a direct consequence of the standard central limit theorem for independent, non identically distributed random variables. 
The covariance function of the limit is obtained as the limit of the two-dimensional Riemann sums
\[
\begin{array}{ll}
\mathbb E[F_{n}(x,y)F_{n}(x',y')] & = \displaystyle{\frac{1}{n^2}\sum_{1\leq k, \ell, \leq n}  \cos \left( \frac{k x}{n}\right)\cos \left( \frac{\ell y }{n}\right)\cos \left( \frac{k x'}{n}\right)\cos \left( \frac{\ell y' }{n}\right)} .
\end{array}
\]
Moreover, if $\partial_1$ and $\partial_2$ denote the partial derivatives in the $x$ and $y$ components, and if we set 
$D_n :=\mathbb E \left[ | F_n(x,y)-F_n(x',y') |^2 \right]$, $D_n^1 :=\mathbb E \left[ | \partial_1 F_n(x,y)-\partial_1 F_n(x',y') |^2 \right]$ and 
$D_n^2 :=\mathbb E \left[ | \partial_2 F_n(x,y)-\partial_2 F_n(x',y') |^2 \right]$, for all $(x,y), (x',y') \in \mathbb R^2$
we have 
\[
\begin{array}{rl}
\displaystyle{ D_n}   =& \displaystyle{\frac{1}{n^2}\sum_{1\leq k, \ell, \leq n} \left| \cos \left( \frac{k x}{n}\right)\cos \left( \frac{\ell y }{n}\right)-\cos \left( \frac{k x'}{n}\right)\cos \left( \frac{\ell y' }{n}\right)\right|^2}\\
\\
& \leq \displaystyle{\frac{2}{n^2}\sum_{1\leq k, \ell, \leq n} \left| \cos \left( \frac{k x}{n}\right)-\cos \left( \frac{k x'}{n}\right)\right|^2 + \left| \cos \left( \frac{\ell y }{n}\right)-\cos \left( \frac{\ell y' }{n}\right)\right|^2} \\
\\
& \leq \displaystyle{\left( \frac{2}{n}\sum_{1\leq k \leq  n} \left( \frac{k}{n}\right)^2 \right) || (x,y)-(x',y')||^2 \leq 2  \, || (x,y)-(x',y')||^2}.
\end{array}
\] 
In the same way,  we have
\[
\begin{array}{rl}
\displaystyle{ D_n^1}   =& \displaystyle{\frac{1}{n^2}\sum_{1\leq k, \ell, \leq n} \frac{k^2}{n^2} \left| \sin \left( \frac{k x}{n}\right)\cos \left( \frac{\ell y }{n}\right)-\sin \left( \frac{k x'}{n}\right)\cos \left( \frac{\ell y' }{n}\right)\right|^2}\\
\\
& \leq \displaystyle{\frac{2}{n^2}\sum_{1\leq k, \ell, \leq n} \frac{k^2}{n^2}  \left( \left| \sin \left( \frac{k x}{n}\right)-\sin \left( \frac{k x'}{n}\right)\right|^2 + \left| \cos \left( \frac{\ell y }{n}\right)-\cos \left( \frac{\ell y' }{n}\right)\right|^2\right)} \\
\\
& \leq \displaystyle{\left( \frac{2}{n}\sum_{1\leq k \leq  n} \frac{k^4}{n^4}\right)  || (x,y)-(x',y')||^2 \leq 2  \, || (x,y)-(x',y')||^2},
\end{array}
\] 
and the exact same computation yields $D^2_n \leq 2\,  || (x,y)-(x',y')||^2$. Together with the convergence of finite dimensional marginals, the three last estimates provide the desired tighness criterion ensuring the convergence in the $C^1$ topology.
\end{proof}

As noticed in Remark \ref{rem.nonstat} in the introduction, we consider here random trigonometric polynomials in a non-stationary framework. To be able to deal with this non-stationarity in our approach of global universality at the end of the paper, we need to slightly reinforce the above convergence result, by establishing a kind of uniformity in space. This is the object of the next proposition.

\begin{prop}\label{pro.unif}
For any  $0<a<b<1$ and any sequence of couples of integers $(p_n,q_n)$ in the square $[a n, b n]^2$, the stochastic process $(G_n(x,y))_{(x,y)\in [0,\pi]^2}$ defined by 
\[
G_n(x,y):=F_n(p_n \pi+x,q_n \pi+y), \;\; (x,y)\in [0,\pi]^2,
\]
converges in distribution, as $n$ goes to infinity, in the space $C^1([0,\pi]^2)$ towards a stationary Gaussian field $G_{\infty}$ of covariation $\rho((x,y),(x',y')):=\frac 1 4 \sin_c(x-x')
\sin_c(y-y')$. 
\end{prop}

\begin{proof}
First of all, the tightness criterion used in the proof of Proposition \ref{conv.C1} applies in the same way since the final bound is expressed only in terms of $\|(x,y)-(x',y')\|_2^2$, so that $p_n$ and $q_n$ play no role here. Thus, one is only left to consider the convergence of the covariations. Setting 
\[
\rho_n(x,x',p):= \frac 1 n \sum_{1\le k \le n} \cos\left(\frac{k}{n}(x+p\pi)\right)\cos\left(\frac{k}{n}(x'+p\pi)\right),
\]
we have 
$
\E\left[F_n(p_n\pi+x,q_n\pi+y)F_n(p_n\pi+x',q_n\pi+y')\right] = \rho_n(x,x',p_n) \rho_n(y,y',q_n).
$
By symmetry, it is enough to investigate the first factor, which can be rewritten as
\[
 \rho_n(x,x',p_n) =\frac 1 {2n} \sum_{1\le k \le n} \cos\left(\frac{k}{n}(x+x'+2 p_n\pi)\right)+\frac 1 {2n} \sum_{1\le k \le n} \cos\left(\frac{k}{n}(x-x')\right).
\]
The second term is a Riemann sum converging to the desired sinus cardinal, whereas the first sum is managed by a direct computation to obtain the inequality
\begin{eqnarray*}
\frac{1}{n}\left|\sum_{1\le k \le n} \cos\left(\frac{k}{n}(x+x'+2 p_n\pi)\right)\right| \le \frac{1}{n}\frac{1}{\left|\sin\left(\frac{x+x'}{2n}+\frac{p_n\pi}{n}\right)\right|}.
\end{eqnarray*}
The right hand side of this last equation goes to zero as $n$ goes to infinity. 
Indeed, on the one hand $(x+x')/2n$ goes to zero as $n$ goes to infinity, whereas on the other hand, $\text{dist}\left(p_n/n,\mathbb{Z}\right)$ remains uniformly bounded from below, hence the result.
\end{proof}

Using the same arguments, one can moreover establish the following convergence result which will also be used at the end of proof of the global universality.
\begin{prop}\label{cvgce-Finfini}
Let $(p_n,q_n)$ be couple of integers as in Proposition \ref{pro.unif}, then the process $F_\infty(p_n \pi+\cdot,q_n \pi+\cdot)$ converges in distribution in the $C^1$ topology towards $G_\infty$. 
\end{prop}

Let us go back to the convergence of the random field $(F_{n}(x,y))_{(x,y) \in K}$ in a fixed compact $K \subset \mathbb R^2$ and establish that the limit Gaussian field $(F_{\infty}(x,y))_{(x,y) \in K}$ is non-degenerate in the following sense.
\begin{lem}\label{lem.nondeg}
The limit Gaussian field $F_{\infty}$ obtained in Proposition \ref{conv.C1} is non-degenerate in the sense that almost surely, we have 
\[
\nabla_{(x,y)} F_{\infty} \neq 0, \;\; \hbox{whenever} \,\, F_{\infty}(x,y)=0. 
\]
\end{lem}
\begin{proof} Let us denote by $A:=\{x =0\} \cup \{y =0\}$ the axes of $\mathbb R^2$ and consider a compact set $K \subset \mathbb R^2 \backslash A$. The fact that the field $F_{\infty}$ is non-degenerate on $K$ is a consequence of Bulinskaya Lemma, see e.g. Proposition 6.11 of \cite{azais}. The only delicate point to check is that the Gaussian vector 
$V=(F_{\infty} , \partial_1 F_{\infty} , \partial_2 F_{\infty} )$ admits a uniformly bounded density on $K$. A necessary and sufficient condition ensuring this fact is that the determinant of the covariance matrix $\Gamma_V$ of $V$ is stricly positive on the compact $K$, and thus uniformly bounded from below. The covariance matrix $\Gamma_V$ of $V$ is a Gram matrix namely if $\langle ,  \rangle$ denotes the standard Hilbert scalar product in $L^2([0,1])$, we have
\[
\Gamma_V = \left( \begin{array}{lll} 
\langle f,f \rangle  & \langle f,g \rangle & \langle f,h \rangle  \\
\langle f,g \rangle  & \langle g,g \rangle & \langle g,h \rangle  \\
\langle f,h \rangle  & \langle g,h \rangle & \langle h,h \rangle  
\end{array} \right) ,
\]
 where 
 \[
 f(s):= \cos( s x) \cos(sy) , \quad  g(s):= -s \sin( s x) \cos(sy), \quad h(s):= -s \cos( s x) \sin(sy).
 \]
 The determinant of this Gram matrix vanishes if and only if the above functions of $s$ are proportional, which only occurs on the axes $\{x =0\}$ or $\{y=0\}$, hence the result. Let us now consider the case of the axes. Let us first remark that the random variable $F_{\infty}(0,0)$ is a standard Gaussian variable so that $F_{\infty}(0,0)\neq 0$ almost surely. Next, on the axis $\{x =0, y \neq 0\} $, the limit process $(F_{\infty}(0,y))_{y \in \mathbb R}$ is nothing but the limit Gaussian process associated to the univariate trigonometic polynomials 
\[
F_n(0,y)= \frac{1}{\sqrt{n}} \sum_{1\leq  \ell \leq n} b_{ \ell} \cos \left( \frac{\ell y }{n}\right),
\]
where the variables $b_{\ell} = \frac{1}{\sqrt{n}} \sum_{k=1}^n a_{k , \ell}$ are independent and identically distributed, their common law being centered and with unit variance. As above, the covariance matrix of  $(F_{\infty}(0,y), \partial_y F_{\infty}(0,y))$ is also a Gram matrix whose determinant only vanishes at the origin, hence is unifomly bounded from below on any compact set of $\{x =0, y \neq 0\} $. Naturally the same reasoning holds on the set 
$\{y=0, x \neq 0\}$.
\end{proof}

\begin{rem}\label{rem.nondege}
Note that the above arguments actually also ensure the non-degeneracy of the stationary limit field $G_{\infty}$ appearing in Propositions \ref{pro.unif} and \ref{cvgce-Finfini}.
\end{rem}

\subsection{Continuity of the nodal length}\label{sec.continuity}
In this section, we establish that the functional that associates to a function $f : \mathbb R^2 \to \mathbb R$ the length of its nodal set, or more generally  its $d-1$ dimensional volume if  $f : \mathbb R^{d} \to \mathbb R$, is continuous with respect to the $C^1$ topology on compact sets. 
Let us be more precise and consider the space $E:=C^1(\mathbb R^d, \mathbb R)$ endowed with the $C^1$ topology associated to the family of semi-norms $|| \cdot ||_K$ :
\[
||f||_K := \sup_K \left(  |f | + \sum_{i=1}^d |\partial_i f|\right), \; \hbox{$K$ compact subset of $\mathbb R^d$}.
\]
Given such a compact $K \subset \mathbb R^d$, we will say that $f \in E$ is non-degenerate on $K$ if 
\[
\nabla_x  f \neq 0  \quad \text{whenever} \quad x \in Z_K(f).
\]
If $A \subset \mathbb R^d$ is a measurable set, we will denote by $H_{d-1}(A)$ with values in $[0, +\infty]$ its $(d-1)$--dimensional Hausdorff measure, so that the object of interest here is the continuity in $f$ of the nodal volume $v_K(f) := H_{d-1} (Z_K(f))$.

\begin{thm}\label{thm.conti}
Let $K \subset \mathbb R^d$ be a compact set and let $(f_n)_{n \geq 1}$ be a sequence of functions in $E$ which converges to a function $f \in E$ in the $C^1$ topology on $K$. If $f$ is non-degenerate on $K$, then the volumes $v_K(f)$ and $v_K(f_n)$, $n$ sufficiently large, are finite and we have
\[
\lim_{n \to +\infty} v_K(f_n) =v_K(f).
\]
\end{thm}

\begin{proof}[Proof of Theorem \ref{thm.conti}]
We first need to introduce some notations.
For a non-degenerate function $f$, we denote by  $\sigma(x)=\sigma_f(x)$ the index of the first non vanishing component of the gradient at $x$, namely
\[
\sigma(x)=\sigma_f(x):=\inf\{1\leq i \leq d, \, \partial_i f(x) \neq 0\}.
\]
If $x=(x^1, \ldots, x^d)$ and $1\leq i \leq d$, we will write 
\[
\pi_i(x) = \check{x}^i:=(x^1, \dots, x^{i-1}, x^{i+1}, \ldots, x^d).
\]
Finally, if $y \in \mathbb R^d$ and $\delta, \varepsilon>0$, $R^i(y,\delta,\varepsilon)$ will denote the following open rectangle
\[
R^i(y,\delta,\varepsilon):=\{x \in \mathbb R^d, |x^i - y^i|  < \delta, |x^{\ell}- y^{\ell} | < \varepsilon, \, 1 \leq \ell \leq d, \, \ell \neq i\}.
\]
Let us first prove the following lemma, which ensures that under the hypotheses of Theorem \ref{thm.conti} and for $n$ sufficiently large, the zeros of $f_n$ are located in a neighborhood of the zeros of $f$. 

\begin{lem} \label{lem.conti}
Let $(f_n)_{n \geq 1}$ be a sequence of functions in $E$ which converges to a function $f \in E$ with respect to the $C^1$ topology on the compact $K$. 
For all $\varepsilon>0$ and for $n$ sufficiently large, we have 
\[
Z_K(f_n) \subset Z_K(f,\varepsilon) := \{ x \in \mathbb R^d, d(x, Z_K(f)) \leq \varepsilon \}.
\]
\end{lem}

\begin{proof}[Proof of Lemma \ref{lem.conti}]
By contradiction, let us suppose that there exists $\varepsilon>0$ such that  for all $N \geq 1$, there exists $n \geq N$ and $x_n \in Z_K(f_n)$ such that $d(x_n, Z_K(f))> \varepsilon$. Since the sequence $(x_n)_{n \geq 1}$ takes values in the compact set $K$, one could then extract a converging subsequence $(x_{n_k})_{k \geq 1}$, converging to some $x_{\infty} \in K$ with  $d(x_{\infty}, Z_K(f))\geq \varepsilon$. But 
\[
\begin{array}{ll}
|f(x_{\infty})| & = \displaystyle{|f(x_{\infty})-f_{n_k}(x_{n_k})| = |f(x_{\infty})-f_{n_k}(x_{\infty}) + f_{n_k}(x_{\infty})-f_{n_k}(x_{n_k})|} \\
\\
& \leq \displaystyle{ \sup_{x \in K} | f(x) - f_{n_k}(x) | + \sup_{x \in K} | f_{n_k}'(x) | \times |x_{\infty} -x_{n_k}|  },
\end{array}
\]
which would go to zero as $k$ goes to infinity because $f_n$ converges to $f$ in the $C^1$ topology on $K$, hence the contradiction between the two assertions $f(x_{\infty})=0$ and $d(x_{\infty}, Z_K(f))\geq \varepsilon$.
\end{proof}

Let us go back to the proof of Theorem \ref{thm.conti} and consider the evaluation application from $ E \times \mathbb R^d$ to $\mathbb R$ defined by 
\[
F(h,x):=h(x).
\]
By hypothesis, since the function $f$ is non-degenerate on $K$, if $x_0=(x_0^1, \ldots, x_0^d) \in Z_K(f)$, we have $F(f,x_0)=0$ and there exists $1\leq i=\sigma_f(x_0) \leq d$ such that $\partial_{x^i} F$  is invertible at $(f,x_0)$. By the implicit function theorem, there exists $\varepsilon_0>0$, $\delta_0>0$ and a  function $X_0 : E \times \mathbb R \to \mathbb R$ of class $C^1$ such that 
\begin{equation}\label{eq.1}
h(x) = 0 \; \Longleftrightarrow \; x^i=X_0(h,\check{x}^i) \;\;\text{for all} \; \left \lbrace \begin{array}{l} x \in R^i(x_0,2\delta_0,2\varepsilon_0), \\ || h-f|| < 2\varepsilon_0. \end{array} \right.
\end{equation}
From the covering of the compact nodal set $Z_K(f)$ by the union of open sets of the type $R^i(x_0,\delta_0,\varepsilon_0)$, one can extract a finite covering. Namely there exists a positive integer $m$ and for all $1 \leq j \leq m$, there exists $x_j \in Z_K(f)$ as well as $\varepsilon_j>0$ and $\delta_j>0$ such that
\begin{equation}\label{eq.recouvre}
Z_K(f) \subset  \bigcup_{j=1}^m V_j, \quad \hbox{where} \quad V_j:=   R^{\sigma(x_j)}(x_j,\delta_j,\varepsilon_j).
\end{equation}
For all $1\leq j \leq m$, if $k=\sigma_f(x_j)$, we have a similar identification to the one given by Equation \eqref{eq.1}, namely in a neighborhood of $(f,x_j)$
\begin{equation}\label{eq.2}
h(x)=0 \Longleftrightarrow x^k=X_j(h,\check{x}^k), \;\;\text{for all} \; \left \lbrace \begin{array}{l} x \in B^k(x_j,2\delta_j,2\varepsilon_j), \\ || h-f|| < 2\varepsilon_j, \end{array} \right.
\end{equation}
where the application $X_j : E \times \mathbb R \to \mathbb R$ is of class $C^1$.
In particular, setting $h=f$, we get that
if $J=\{j_1, \ldots, j_r\} \subset \{1, \ldots, m\}$ and $\bigcap_{j \in J} V_j \neq \emptyset$, the intersection
\[
\Gamma_{J}= Z_K(f) \bigcap \left( \bigcap_{j \in J}  \overline{V}_{j}\right)
\] 
identifies with a parametrized hypersurface whose finite volume is given by the classical formula
\begin{equation}\label{eq.5}
H_{d-1}(\Gamma_{J}) = \displaystyle{\int_{E_J} \sqrt{1+ |\nabla X_{j_1}(f,y)|^2}dy, }
\end{equation}
where the integration is performed on the compact rectangle
\[
E_J := \pi_{j_1} \left( \bigcap_{j \in J} \overline{V}_{j} \right).
\]
Taking care of the overlapping, the finite total volume of the nodal set is then given by the celebrated Poincar\'e formula
\begin{equation}\label{eq.length0}
v_K(f)  = \sum_{\emptyset \neq J \subset \{1, \ldots, m\}} (-1)^{|J|} H_{d-1}(\Gamma_J).
\end{equation}

\begin{figure}[ht]\label{fig.3}
\begin{center}
\includegraphics[scale=0.35]{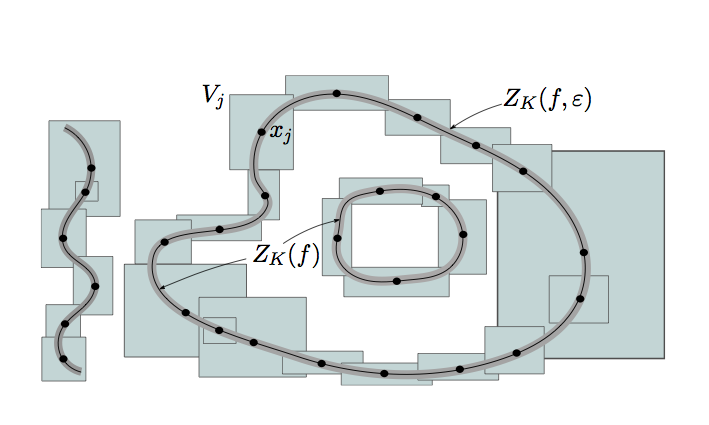}
\end{center}
\caption{Finite covering of the compact nodal set.} 
\end{figure}

Let us now emphasize the fact that in Equation \eqref{eq.recouvre}, the union not only contains the nodal set $Z_K(f)$, but there exists $\varepsilon>0$ small enough such that this union contains a $\varepsilon-$neighborhood of the latter :
\[
Z_K(f,\varepsilon) \subset \bigcup_{j=1}^m V_j.
\]
By Lemma \ref{lem.conti}, we get that for $n$ large enough $Z_K(f_n) \subset Z_K(f, \varepsilon) \subset \bigcup_{i=j}^m V_j$ and thus
\[
Z_K(f_n) = \bigcup_{j=1}^m \left[ Z_K(f_n) \cap  \overline{V}_j  \right].
\]
From the equivalence \eqref{eq.2} given by the implicit function theorem, as above, we get that
if $J=\{j_1, \ldots, j_r\} \subset \{1, \ldots, m\}$ and $\bigcap_{j \in J} V_j \neq \emptyset$, the intersection
\[
\Gamma_{J}^n= Z_K(f_n) \bigcap \left( \bigcap_{j \in J}  \overline{V}_{j}\right)
\] 
also identifies with a parametrized hypersurface whose volume is given by 
\begin{equation}\label{eq.6}
H_{d-1}(\Gamma_{J}^n) = \displaystyle{\int_{E_J} \sqrt{1+ |\nabla X_{j_1}(f_n,y)|^2}dy.}
\end{equation}
By the Poincar\'e formula, we have similarly
\begin{equation}\label{eq.length1}
v_K(f_n)  = \sum_{\emptyset \neq J \subset \{1, \ldots, m\}} (-1)^{|J|} H_{d-1}(\Gamma_J^n),
\end{equation}
so that, comparing to Equation \eqref{eq.length0}, we get
\[
\left| v_K(f)-v_K(f_n) \right| \leq \sum_{\emptyset \neq J \subset \{1, \ldots, m\}}  | H_{d-1}(\Gamma_{J})-H_{d-1}(\Gamma_{J}^n)|.
\]
The right hand side of the last equation goes to zero as $n$ goes to infinity because, from Equations \eqref{eq.5} and \eqref{eq.6}, for any non-empty subset $J=(j_1,\ldots, j_r)$ of $\{1, \ldots, m\}$, we have 
\[
\left| H_{d-1}(\Gamma_{J})- H_{d-1}(\Gamma_{J}^n) \right| \leq  \displaystyle{\int_{E_J} \left| \sqrt{1+ |\nabla X_{j_1}(f,y)|^2}- \sqrt{1+ |\nabla X_{j_1}(f_n,y)|^2}\right|dy, }
\]
and  the difference $\nabla X_{j_1}(f,y) - \nabla X_{j_1}(f_n,y)$ goes to zero uniformly on $E_J$, since the function $X_{j_1}$ is $C^1$ and since the sequence $f_n$ converges to $f$ in the $C^1$ topology on $K$.

\end{proof}

\subsection{Local universality}\label{Locuniv}
Let $K \subset \mathbb R^2$ a compact set. 
Combining  Proposition \ref{conv.C1} and Lemma \ref{lem.nondeg}, we get that as $n$ goes to infinity, the field $(F_n(x,y))_{x \in K}$ converges with respect to the $C^1$ topology on $K$ to a the non-degenerate limit field $(F_{\infty}(x,y))_{x \in K}$.  The announced local universality result is then a direct consequence of the continuous mapping theorem together with the continuity of the nodal length established in Theorem \ref{thm.conti}. 

\begin{thm}\label{theo.local}
Let $K \subset \mathbb R^2$ a compact set, then as $n$ goes to infinity, the length $\ell_K(F_n)$ of the nodal set converges in distribution to $\ell_K(F_{\infty})$.
\end{thm}

\section{Global universality}\label{sec.global}

We now turn to the proof of Theorem \ref{theo.main} on the universality of the mean nodal length at the macroscopic level.

\subsection{The Gaussian case} \label{sec.gaussian}
In this section, we consider the Gaussian case, namely we assume that all the  entries $a_{k,l}$ are independent standard Gaussian variables. In this situation, the expectation of the nodal length $\ell_K(f_n)$ can be explicitely computed thanks to celebrated Kac--Rice formula, since both $f_n$ and its derivative have explicit densities.  
\begin{lem}\label{asympt} 
For $(x,y) \in \mathbb R^2$, the Gaussian vector  $(f_n(x,y), \frac{\partial f_n}{\partial x}((x,y), \frac{\partial f_n}{\partial y}((x,y))$ is centered with explicit covariance $\Sigma=(\Sigma_{i j})_{1 \leq i,j\leq 3}$  given by 
\[
\begin{array}{lll}
\Sigma_{11}=A_n(x)A_n(y),  & \Sigma_{22}= C_n(x)A_n(y), &  \Sigma_{33}=A_n(x)C_n(y) ,\\
\Sigma_{12}=-B_n(x)A_n(y), &  \Sigma_{13}= -A_n(x)B_n(y), & \Sigma_{23}= B_n(x)B_n(y),\\
\end{array}
\]
where
\[ 
A_n(\centerdot):=  \sum_{1\le k \le n} \cos^2(k\centerdot),\;\; 
B_n(\centerdot):=  \sum_{1\le k \le n} k\sin(k\centerdot)\cos(k\centerdot), \;\; 
 C_n(\centerdot):=  \sum_{1\le k \le n} k^2\sin^2(k\centerdot). 
\]
\end{lem}
Note that the sums $A_n$, $B_n$ and $C_n$ appearing in Lemma  \ref{asympt} can actually be written as simple combinations of trigonometric functions. 
For example, the next lemma can be found in \cite{wilkins}.
\begin{lem}\label{asympto} We have
\[
\begin{array}{ll}
\displaystyle{4A_n(x)} & =\displaystyle{(2n+1)g_0+g_1}, \\
\displaystyle{8B_n(x) } & =  \displaystyle{(2n + l)^2h_0  +  (2n +  l)h_1+h_2}, \\
\displaystyle{48C_n(x)} & = \displaystyle{ (2n+  l)^3 k_0  +  (2n +  l)^2 k_1+(2n+  l) k_2  +  k_3},
\end{array}
\]
where, setting  $z:= (2n+1)x$, and $f(x):=\mbox{csc}(x)-x^{-1}$, the functions $g_i$, $h_i$ and $k_i$ are defined as
\begin{align*}
g_0(x):=&1+z^{-1}\sin z, \; g_1(x)= -2+ f(x)\sin z,\\
h_0(x):= &-z^{-1}\cos z + z^{-2} \sin z, \; h_1(x)= -f(x) \cos z, \; h_2(x)=-f'(x)\sin z, \\
k_0(x):=&1-3z^{-1}\sin z -6z^{-2}\cos z +6z^{-3}\sin z,\\
k_1(x):=&-3f(x)\sin z, \; k_2(x)=6f'(x)\cos z -1, \; k_3(x)= 3f"(x)\sin z.
\end{align*}
\end{lem}
It is remarkable that, conditionally to the event $f_n=0$, the partial derivatives of $f_n$ are independent Gaussian random variables.
\begin{lem}\label{condidis} Given $f_n=0$, the conditional distribution of $\left( \frac{\partial f_n}{\partial x}, \frac{\partial f_n}{\partial y}\right)$ is
$$\left( \frac{\partial f_n}{\partial x}, \frac{\partial f_n}{\partial y}\right) \sim \mathcal{N}\left( 0, \frac{1}{\Sigma_{11}} \left(
\begin{array}{cc}
\Sigma_{11}\Sigma_{22}-\Sigma^2_{12} & 0 \\ 
0 & \Sigma_{11}\Sigma_{33}-\Sigma^2_{13}
\end{array} \right)\right).$$
\end{lem}
\begin{proof}
Conditionally to the event $f_n=0$, the conditional covariance matrix $\Sigma_{f_n=0}$ of the gradiant vector $\nabla f_n=\left( \frac{\partial f_n}{\partial x}, \frac{\partial f_n}{\partial y}\right)$   is given by 
\[
\begin{array}{ll}
\Sigma_{f_n=0} & = \textrm{Var}(\nabla f_n)-\textrm{Cov}(\nabla f_n,f_n)[\textrm{Var}(\nabla f_n)]^{-1}[Cov(\nabla f_n,f_n)]^T \\
\\
& =\displaystyle{\frac{1}{\Sigma_{11}} \left(
\begin{array}{cc}
\Sigma_{11}\Sigma_{22}-\Sigma^2_{12} & \Sigma_{11}\Sigma_{23}-\Sigma_{12}\Sigma_{13} \\ 
\Sigma_{11}\Sigma_{23}-\Sigma_{12}\Sigma_{13} & \Sigma_{11}\Sigma_{33}-\Sigma^2_{13}
\end{array} \right).}
\end{array}
\]
The result thus follows from the fact that $\Sigma_{11}\Sigma_{23}-\Sigma_{12}\Sigma_{13}=0$.
\end{proof}

We are now in position to explicitely compute the expectation of the length of nodal curve associated to the random trigonometric polynomial $f_n(x,y)$. 

\begin{thm}\label{Casgaussien} Let $(a_{k,l})_{k,l\ge 1}$ be a sequence of independent standard, centered, Gaussian variables and consider the associated random trigonometric polynomial $f_n(x,y)$ defined by Equation \eqref{def.poly}. Then, as $n$ tends to infinity, we have
\[
\mathbb E[ \ell_{[0, \pi]^2}(f_n) ] \sim \frac{(2n+1) \pi^2}{4 \sqrt{3}}.
\]
\end{thm}

\begin{proof}
By Kac--Rice formula, the expectation of the nodal length is equal to
\hspace{-1cm}
\begin{eqnarray*}
\mathbb E[ \ell_{[0, \pi]^2}(f_n) ]  &=&\displaystyle \iint_{[0,\pi]^2}\E \left( \sqrt{\left(\frac{\partial f_n}{\partial x}\right)^2+\left(\frac{\partial f_n}{\partial y}\right)^2} \big| f_n(x,y)=0\right) p_{f_n(x,y)}(0)dxdy \notag \\
&=& \displaystyle \iint_{[0,\pi]^2} \frac{1}{\sqrt{2\pi}\Sigma^{1/2}_{11}}\E \left( \sqrt{\left(\frac{\partial f_n}{\partial x}\right)^2+\left(\frac{\partial f_n}{\partial y}\right)^2} \big| f_n(x,y)=0\right) dxdy, \label{expect}
\end{eqnarray*}
where $p_{f_n(x,y)}$ is the density function of $f_n(x,y)$.
From Lemma \ref{condidis}, we have
$$\E \left( \sqrt{\left(\frac{\partial f_n}{\partial x}\right)^2+\left(\frac{\partial f_n}{\partial y}\right)^2} \big| f_n(x,y)=0\right)=\sqrt{\frac{\Sigma_{11}\Sigma_{22}-\Sigma^2_{12}}{\Sigma_{11}}}\E\left[  \sqrt{Z^2_1+Z^2_2(x,y)}\right],$$
where
$$(Z_1,Z_2((x,y) )\sim  \mathcal{N}\left( 0,  \left(
\begin{array}{cc}
1&0 \\ 
0 & \frac{\Sigma_{11}\Sigma_{33}-\Sigma^2_{13}}{\Sigma_{11}\Sigma_{22}-\Sigma^2_{12} }
\end{array} \right)\right).$$
Now using the explicit formulas of Lemmas \ref{asympt} and \ref{asympto}, as $n$ goes to infinity, we have 
\[
\frac{\sqrt{\Sigma_{11}\Sigma_{22}-\Sigma^2_{12}}}{\Sigma_{11}} \sim \frac{2n+1}{\sqrt{12}},
\]
and the distribution of $(Z_1,Z_2(x,y))$ converges to the one of a standard two-dimensional normal variable $Z=(Z_1,Z_2)$. Substituting these estimates in the above integral expression of $\mathbb E[ \ell_{[0, \pi]^2}(f_n) ]$, yieds, as $n$ tends to infinity,
\[
\mathbb E[ \ell_{[0, \pi]^2}(f_n) ]\sim \frac{2n+1}{\sqrt{24\pi}}\iint_{[0,\pi]^2} \E\sqrt{Z^2_1+Z^2_2} dxdy.
\]
Since $\sqrt{Z^2_1+Z^2_2}$ has the standard Rayleigh distribution, its expectation is equal to $\sqrt{\pi/2}$, which implies the statement of the theorem.
\end{proof}

\subsection{Moment control}\label{sec.moment}

In the above Theorem \ref{theo.local}, we proved that given a compact set $K \subset \mathbb R^2$ and as $n$ goes to infinity, the microscopic length $\ell_K(F_n)$ of the nodal set of the normalized trigonometric polynomial converges in distribution to $\ell_K(F_{\infty})$. The object of this subsection is to establish a uniform upper bound for the expectation of this microscopic length, uniform in both the degree $n$ and in the compact $K$. More precisely, taking care of the change of scale on the length of the nodal set, the mean macroscopic nodal length can be written as the sum
\begin{equation}\label{eq.sumlocal}
\frac{\mathbb E[ \ell_{[0,\pi]^2}(f_n) ]}{n} = \frac{\mathbb E[ \ell_{[0,n\pi]^2}(F_n) ]}{n^2} =  \frac{1}{n^2}\sum_{0 \leq k,l\le n-1}\E\left[l_{n,k,l}\right],
\end{equation}
where $l_{n,k,l}$ denotes the length of the nodal set associated to $F_n(x,y)$ inside the square $[k\pi,(k+1)\pi]\times[l\pi,(l+1)\pi]$. We shall prove the following uniform upper bound. 

\begin{prop} There exists $\alpha>0$ and $C>0$ such that 
\begin{equation}\label{.unibound}
\sup_{n \geq 1}\sup_{ 0\leq k, l\leq n-1} \E\left[l_{n,k,l}^{1+\alpha}\right]\leq C.
\end{equation}\label{pro.unibound}
\end{prop}

\subsubsection{Geometric considerations}
In this first subsection, we prove two elementary and purely geometric results. Both results relate the length of a smooth curve drawn in a unit square to the number of its intersections with some prescribed lines. As a corollary, we derive an a priori estimate for the microscopic length of a trigonometric polynomial in a unit square. Both proofs use the so-called probabilistic method saying that a random variable $X$ such that $\E(X)\ge c$ admits at least one realization $\omega$ such that $X(\omega)\ge c$.
\begin{rem}\label{Remgeo}
At first glance, one might be tempted to use the Crofton formula in order to relate the length of the nodal domain of a trigonometric polynomial, to the number of its intersections with some random lines. Nevertheless, such an approach faces two major obstructions. On the one hand, when the slope of such a line is irrational, then when restricting the bivariate trigonometric polynomial to it, the resulting random function is not anymore polynomial. 
For this reason, in the next Theorem \ref{Geo2}, we relate the length of the nodal set to its number of intersections with vertical or horizontal lines, which then allow us to derive a deterministic upper bound on the nodal length. On the other 
hand, since the nodal set is random, the lines intersecting it are also generically random. This randomness dependence is hard to manage when performing characteristic functions computations since we loose the structure of independant summands. This is why we prove the next Theorem \ref{Geo1} just below in order to ``force'' the lines to go through deterministic points on which the independance of summands is preserved and the characteristic functions method applicable.
\end{rem}

\begin{thm}\label{Geo1}
There exists an absolute constant $c>0$ such that for any unit square $\mathcal{S}$ with corners $A,B,C,D$ and any $C^1$ curve $\mathcal{C}$ inside $\mathcal{S}$ with length $l$, one may find a straight line $\mathcal{L}$ such that:
\begin{itemize}
\item[i)] $\{A,B,C,D\} \cap \mathcal{L} \neq \emptyset$,
\item[ii)] $ \#\{ \mathcal{L}\cap \mathcal{C} \} \ge c l$.
\end{itemize}
\end{thm}
\begin{proof} Using the probabilistic method, we will actually establish the above result with $c=1/4$. On a given probability space $(\Omega,\mathcal{F},\mathbb{P})$, we denote by $P$  a random point inside the square with uniform distribution. Set
\[
X_{\mathcal{C}}:=\#\left\{\mathcal{C} \cap \left((AP)\cup (BP) \cup (CP) \cup (DP)\right)\right\}.
\]
Then the result follows if one can show that $\E(X_{\mathcal{C}}) \ge c l$. Indeed, in this case, there exists one realization of the random variable such that $X_{\mathcal{C}}(\omega)\ge c l$. Notice that, since $\mathcal{C}$ is assumed to be $C^1$, it is rectifiable. Hence, one might try to seek for $(\mathcal{C}_p)_{p \geq 1}$ a sequence of polygonal lines $\mathcal{C}_p$ such that:
\[
\begin{array}{ll}
\text{(i)}& \forall p \ge 1,\,\E[X_{\mathcal{C}}]\ge \E[X_{\mathcal{C}_p}],\\
\text{(ii)}& \ell(\mathcal{C}_p) \to \ell(\mathcal{C}),\\
\text{(iii)} & \E[X_{\mathcal{C}_p}] \ge \frac{1}{4}\ell(\mathcal{C}_p).
\end{array}
\]
Assume that the curve $\mathcal{C}$ is parametrized by two functions of class $C^1$, that is to say $\mathcal{C}=\left\{ (x(t),y(t))\,\left|\right. t\in [0,1] \right\}$ and consider the polygonal line $\mathcal{C}_p$ interpolating between the points $(x(\frac{k}{p}),y(\frac{k}{p}))$, for $0 \leq k \leq p$. At this stage, we notice that (i) is a consequence of connexity and (ii) proceeds from the fact that $\mathcal{C}$ is rectifiable. Thus, one is only left to establish (iii).
By the linearity of the expectation, without loss of generality, we may just consider the case when $\mathcal{C}$ is the segment $IJ$ contained in the domain $OCD$, see Figure \ref{linecorner} below. If it is not the case, then we can always split it in two segments respectively contained in the domains $OCD$ and $ABC$ respectively.
Note that the point $I$ is on the left of $J$. Assume that $J$ is higher than $I$. Since for each line $(AP)$ (or $(BP),(CP),(DP)$), there is at most one intersection point with $\mathcal{C}$,
\begin{align*}
\E[X_{\mathcal{C}}] & = \displaystyle \mathbb{P} \{(AP) \cap \mathcal{C} \neq \emptyset  \}+\mathbb{P} \{(BP) \cap \mathcal{C} \neq \emptyset  \}+\mathbb{P} \{(CP) \cap \mathcal{C} \neq \emptyset  \}+\mathbb{P} \{(DP) \cap \mathcal{C} \neq \emptyset  \} \\
&= \displaystyle \frac{\lambda_2(AA_1A_2)+\lambda_2(BB_1B_2)+\lambda_2(CC_1C_2)+\lambda_2(DD_1D_2)}{\lambda_2(ABCD)}\\
&= A_1A_2+B_1B_2+C_1C_2+D_1D_2,
\end{align*}
where $\lambda_2$ stands for the area or two-dimensional Lebesgue measure and $A_1,A_2$ are the intersections between $AI,AJ$ and $CD$. 

From $I$, draw a line parallel to $CD$ which intersects $AA_2$ at $I_1$; similarly, draw a line parallel to $(AD)$ which intersects $CC_2$ at $I_2$. Now, draw the rectangle $II_2I_3I_1$. It is easy to check that the point $J$ must lie inside this rectangle. Therefore,
$$A_1A_2+C_1C_2 \geq II_1+II_2 \geq II_3 \geq IJ.$$
Here we use a simple observation that: the largest distance between two points in a rectangle is the length of the diagonal. Then it implies $\E (X_{\mathcal{C}}) \ge IJ=\mbox{length}(\mathcal{C})$.

 \begin{figure}[ht]
\centering
\begin{tikzpicture}[scale=2]
\draw[thick] (0,0) rectangle (2,2) ;
\draw[dotted] (0,0) -- (2,2);
\draw[dotted] (0,2) -- (2,0);
\draw (0.5,0.3)--(1.2,0.6);
\draw (2,0)--(0,0.4);
\draw (2,0)--(0,1.5);
\draw (0,2)--(0.588,0);
\draw (0,2)--(1.714,0);
\draw[dashed](0.5,0.3)--(1.457,0.3)--(1.457,1.125)--(0.5,1.125)--(0.5,0.3);
\draw (0,0) node[left]{D};
\draw (0,2) node[left]{A};
\draw (2,0) node[right]{C};
\draw (2,2) node[right]{B};
\draw (0.5,0.25) node[left]{I};
\draw (1.2,0.62) node[right]{J};
\draw (0,0.4) node[left]{$C_1$};
\draw (0,1.5) node[left]{$C_2$};
\draw (0.588,-0.1) node{$A_1$};
\draw (1.714,-0.1) node{$A_2$};
\draw (1,1) node[right]{O};
\draw(1.457,0.3) node[right]{$I_1$};
\draw (0.5,1.125) node[left]{$I_2$};
\draw (1.457,1.125) node[right]{$I_3$};
\end{tikzpicture}
\caption{A segment case.} \label{linecorner}
\end{figure}
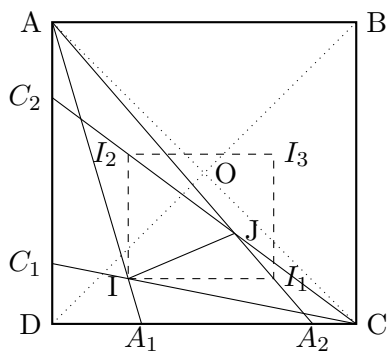

Otherwise, if $I$ is higher than $J$ we make an analoguous reasoning by simply considering the two triangles $BB_1B_2$ and $DD_1D_2$.
\end{proof}
\begin{thm}\label{Geo2}
There exists an absolute constant $c>0$ such that for any unit square $\mathcal{S}$ with corners $A,B,C,D$ and any $C^1$ curve $\mathcal{C}$ inside $\mathcal{S}$ with length $l$, one may find an horizontal or vertical straight line $\mathcal{L}$ such that $\# \{\mathcal{L}\cap \mathcal{C} \} \ge c l$.
\end{thm}
\begin{proof}
Here, we use again the probabilistic method and the piecewise linear approximation to prove the claimed result for $c=1/2$. Let us just consider the curve $\mathcal{C}$ as a segment $IJ$. Choose uniformly a horizontal line inside the square (i.e. choose uniformly a point on $AD$ and draw a horizontal line from this point), define $X_1$ as the number of intersection points between this line and $IJ$. Similarly, choose uniformly a vertical line inside the square and define $X_2$. Then clearly,
\[
\E X_1 +\E X_2 =I_1J_1+I_2J_2 \geq IJ,
\]
where $I_1J_1$ and $I_2J_2 $ are the projections of $IJ$ on $AD$ and $CD$. 
Therefore, there exist a horizontal line and a vertical one such that the total number of intersection points with the nodal curve is at least $l$. This yields the statement of the theorem.
\end{proof}

We can now derive the announced a priori estimate on the microscopic nodal length.

\begin{cor}\label{lem-apriori}
Suppose that $Q(x,y)$ is any trigonometric polynomial of degree $n$ and denote by $l_{k,l}$ the length of the nodal line of $n^{-1} Q(\frac{x}{n},\frac{y}{n})$ in $[k\pi,(k+1)\pi]\times[l\pi,(l+1)\pi]$. Then we have
\begin{equation}\label{eq-apriori}
l\le \frac{2n}{c} 
\end{equation} 
\end{cor}

\begin{proof}
Thanks to Theorem \ref{Geo2}, there exists a vertical or horizontal line having at least $l/2$ intersection points with the nodal curve. Otherwise, restricted on this line, $Q(x,y)$ becomes a trigonometric polynomial with only one variable; so it has at most $n$ roots over any unit interval. Then the result follows.
\end{proof}

\subsubsection{Small ball estimate}
In this section, we show that the uniform upper bound stated in Proposition \ref{pro.unibound} actually reduces to a small ball estimate for the rescaled polynomial $F_n$ at well chosen lattice points.  
To do so, let us first recall some standard number theory considerations which will be used throughout the sequel. Let $n$ be any positive integer and let $p \in \mathbb N$. We shall denote by $\text{ord}(p)$ the order of $p$ in the group $\left(\mathbb{Z}/n \mathbb{Z},+\right)$, that is to say, $\text{ord}(p)= \frac{n}{\gcd(p,n)}$. Then we have the two next lemmas.

\begin{lem}\label{bound-order}
$\max(\text{ord}(p),\text{ord}(p+1)) \ge \sqrt{n}.$
\end{lem}
\begin{proof}
Arguing by contradiction, we both assume that $\text{ord}(p)<\sqrt{n}$ and $\text{ord}(p+1)<\sqrt{n}$. We then have $\gcd(p,n)>\sqrt{n}$ and $\gcd(p+1,n)>\sqrt{n}$. However, since $\gcd(p,p+1)=1$ it holds that $\gcd\left(\gcd(p,n),\gcd(p+1,n)\right)=1$ and thus $\gcd(p,n) \gcd(p+1,n)$ divides $n$. This implies $n< \gcd(p,n) \gcd(p+1,n) \le n$ which is a contradiction.
\end{proof}

\begin{lem}\label{Lemma-computation}
For any $1$-periodic function $f$ and any integer $p\ge 1$, 
\begin{equation}\label{Equation-computation}
\frac{1}{n}\sum_{k=1}^n f \left(\frac{kp}{n}\right)=\frac{1}{\text{ord}(p)}\sum_{k=1}^{\text{ord}(p)} f \left(\frac{k}{\text{ord}(p)}\right).
\end{equation}
\end{lem}
\begin{proof}
It is clear that $p/n=q/\text{ord}(p)$ where $gcd(q,\text{ord}(p))=1$. Since the set $q\times \{1,2,\ldots ,\text{ord}(p)\}$ is a complete residue system of modulo $\text{ord}(p)$ and since the function $f$ is $1$-periodic,
\begin{equation*}
\sum_{k=1}^{\text{ord}(p)}  f \left(\frac{kp}{n}\right) = \sum_{k=1}^{\text{ord}(p)}  f \left(\frac{kq}{\text{ord}(p)}\right) =\sum_{k=1}^{\text{ord}(p)}  f \left(\frac{k}{\text{ord}(p)}\right) .
\end{equation*}
The result follows from the fact that one can divide the set $\{1,2,\ldots , n\}$ into $n/\text{ord}(p)$ complete residue systems.
\end{proof}
\paragraph{Towards a small ball problem}

Let us give us $\alpha>0$ to be chosen later. In virtue of Corollary \ref{lem-apriori}, we have

\begin{eqnarray*}
\E(l_{n,k,l}^{1+\alpha})&=& (1+\alpha)\int_0^\infty t^{\alpha} \mathbb{P}\left( l_{n,k,l} > t\right) dt\\
&\le& (1+\alpha) \int_0^{{\frac{2n}{c}}}t^{\alpha} \mathbb{P}\left( l_{n,k,l} > t\right) dt.\\
\end{eqnarray*}
Thus, one is left to estimate the term $\mathbb{P}\left( l_{n,k,l} > t\right)$. To do so, we shall use the content of Theorem \ref{Geo1}. We place ourselves on the square $[k\pi,(k+1)\pi]\times [l\pi,(l+1)\pi]$ and we know that there exists a straightline, say $\mathcal{L}$, such that
\begin{itemize}
\item[(i)] $(k\pi,l\pi)$ or $((k+1)\pi,l\pi)$ or $(k\pi,(l+1)\pi)$ or $((k+1)\pi,(l+1)\pi)$ is on $\mathcal{L}$,
\item[(ii)] the number of roots of $F_n$ restricted to $\mathcal{L}\cap [k\pi,(k+1)\pi]\times [l\pi,(l+1)\pi]$ is greater than $c t$. 
\end{itemize}
Now, in order to fix the ideas, assume that $(k\pi,l\pi)\in \mathcal{L}$ and denote by $(u,v)$ the unit vector leading the straight line $\mathcal{L}$. We set $\phi_n(t)=F_n (k\pi+ t u , l\pi+t v)$ for $t\in [ 0, T ]$ where $T$ is the largest positive number such $(k\pi,l\pi)+t (u,v)$ is inside the square. In particular, a simple application of Pythagore Theorem entails that $T\le \pi \sqrt{2}$. As a result, we know that $\phi_n$ vanishes at least $r=\lfloor c t \rfloor$ times in the interval $[0,\pi \sqrt{2}]$. Let us introduce $a_1$ a root of $\phi'$, $a_2$ a root of $\phi''$, $a_3$ a root of $\phi'''$,... and $a_{r-1}$ a root of $\phi^{(r-1)}$ (which exist by a repeated application of Rolle's Theorem). We may write

$$\phi_n(x_1)=\int_{a_1}^{x_1}\int_{a_2}^{x_1}\cdots\int_{a_{r-1}}^{x_{r-1}} \phi^{(r-1)} ( x_r ) dx_r dx_{r-1} \cdots dx_2.$$

Taking $x_1=0$ and using the triangle inequality, one may deduce the following inequality

\begin{equation}\label{bound-infinite norm}
\left|\phi_n(0)\right|=\left| F_n( k \pi, l\pi)\right| \le \frac{(\pi \sqrt{2})^{r-1}}{(r-1)!} \|\phi^{(r-1)}\|_\infty.
\end{equation}
As  a result, for any $M>0$, we get
\begin{eqnarray}
\mathbb{P}\left(l_{n,k,l}>t\right)&\le& \mathbb{P}\left(\left| F_n( k \pi, l\pi)\right| \le \frac{(\pi \sqrt{2})^{r-1}}{(r-1)!} \|\phi^{(r-1)}\|_\infty\right) \nonumber\\
&\le& \mathbb{P}\left(\left| F_n( k \pi, l\pi)\right| \le M\frac{(\pi \sqrt{2})^{r-1}}{(r-1)!} \right)+\mathbb{P}\left(\|\phi^{(r-1)}\|_\infty>M\right).\label{laborne}
\end{eqnarray}
Recall that we have assumed that $(k\pi,l\pi)$ belongs to $\mathcal{L}$. In the general case, we rather have
\begin{eqnarray*}
\mathbb{P}\left(l_{n,k,l}>t\right)&\le& \mathbb{P}\left(\left| F_n( k \pi, l\pi)\right| \le \frac{(\pi \sqrt{2})^{r-1}}{(r-1)!} \|\phi^{(r-1)}\|_\infty\right)\\
&+&\mathbb{P}\left(\left| F_n( (k+1) \pi, l\pi)\right| \le \frac{(\pi \sqrt{2})^{r-1}}{(r-1)!} \|\phi^{(r-1)}\|_\infty\right)\\
&+&\mathbb{P}\left(\left| F_n( k \pi, (l+1)\pi)\right| \le \frac{(\pi \sqrt{2})^{r-1}}{(r-1)!} \|\phi^{(r-1)}\|_\infty\right)\\
&+&\mathbb{P}\left(\left| F_n( (k+1) \pi, (l+1)\pi)\right| \le \frac{(\pi \sqrt{2})^{r-1}}{(r-1)!} \|\phi^{(r-1)}\|_\infty\right)\\
&\le&\mathbb{P}\left(\left| F_n( k \pi, l\pi)\right| \le M\frac{(\pi \sqrt{2})^{r-1}}{(r-1)!}\right)\\
&+&\mathbb{P}\left(\left| F_n( (k+1) \pi, l\pi)\right| \le M\frac{(\pi \sqrt{2})^{r-1}}{(r-1)!}\right)\\
&+&\mathbb{P}\left(\left| F_n( k \pi, (l+1)\pi)\right| \le M\frac{(\pi \sqrt{2})^{r-1}}{(r-1)!}\right)\\
&+&\mathbb{P}\left(\left| F_n( (k+1) \pi, (l+1)\pi)\right| \le M\frac{(\pi \sqrt{2})^{r-1}}{(r-1)!}\right)
+4 \mathbb{P}\left(\|\phi^{(r-1)}\|_\infty>M\right).\\
\end{eqnarray*}

The last estimate requires the following bound.
\begin{lem}\label{bound-Sobolev}
$$\mathbb{P}\left(\|\phi^{(r-1)}\|_\infty>M\right)\le \frac{C}{M}$$
\end{lem}
\begin{proof}
Let us recall that for any compact $K=[a,b]\times [c,d]$, there exists an absolute positive constant $C_{K}$ (depending only on $b-a$ and $d-c$) such that for any $C^1$ mapping $f$, one gets the inequality:
\begin{equation}\label{Sobolev-ineq}
\sup_{x\in K}|f(x)| \leq C_{K}\left( \int_K f^2(x) dx + \int_a^b \|\nabla f(x)\|^2 dx \right)^{\frac{1}{2}}.
\end{equation}
Setting $K=[k\pi,(k+1)\pi]\times[l\pi,(l+1)\pi]$ and recalling that $(u,v)$ is an unit vector, we first notice that
\[
|\phi^{(r)}(t)|=\left|\sum_{i+j=r} \partial_{1}^i\partial_{2}^j F_n(k\pi+tu,l\pi+tv) u^i v^j\right|
\le \sum_{i+j=r} \sup_{x\in K}\left|\sum_{i+j=r} \partial_{1}^i\partial_{2}^j F_n (x)\right|.
\]
Thus, one is left to bound from above each partial derivatives $\partial_1^i \partial_2^j F_n$ on the compact set $K$. Here, we apply the inequality (\ref{Sobolev-ineq}) and we get
\[
\sup_{x\in K}\left|\partial_1^i \partial_2^j F\right| \leq C_{K}\left|\int_K \left( \left|\partial_1^i \partial_2^j F_n(x)\right|^2
+\left|\partial_1^{i+1} \partial_2^j F_n(x)\right|^2
+ \left|\partial_1^i \partial_2^{j+1} F_n(x)\right| ^2\right) dx\right|^{\frac{1}{2}}.
\]
However, for any couple of indexes $(i,j)$, we have by Fubini and orthogonality of the random variables $\{a_{r,s}\}$:
\begin{eqnarray*}
&&\E\left(\int_K \left(\partial_1^i \partial_2^j F_n(x)\right)^2 dx\right)\\
&=&\frac{1}{n^2}\E\left[\int_{k\pi}^{(k+1)\pi}\int_{l\pi}^{(l+1)\pi} \left|\sum_{r,s\le n-1} \left(\frac{r}{n}\right)^i \left(\frac{s}{n}\right)^j a_{r,s} \cos^{(i)}\left(\frac{rx}{n}\right)\cos^{(j)}\left(\frac{ry}{n}\right)\right]^2 dx dy \right]\\
&=&\frac{1}{n^2}\int_{k\pi}^{(k+1)\pi}\int_{l\pi}^{(l+1)\pi} \E\left[\left|\sum_{r,s\le n-1} \left(\frac{r}{n}\right)^i \left(\frac{s}{n}\right)^j a_{r,s} \cos^{(i)}\left(\frac{rx}{n}\right))\cos^{(j)}\left(\frac{ry}{n}\right)\right|^2\right] dx dy\\
&\le &\frac{1}{n^2}\sum_{r,s\le n-1}\left(\frac{r}{n}\right)^{2i} \left(\frac{s}{n}\right)^{2j} \le 1.
\end{eqnarray*}
One is then left to employ the Markov inequality in order to conclude the proof:
\begin{eqnarray*}
&&\mathbb{P}\left( \|\phi^{(r-1)}\|_\infty \ge M\right)\\
&\le& \mathbb{P}\left(\sum_{i+j=r} \sup_{x\in K}\left|\partial_1^i \partial_2^j F_n\right| \ge M\right)\\
&\le& \frac{1}{M}\sum_{i+j=r} \E\left[\sup_{x\in K}\left|\partial_1^i \partial_2^j F_n\right|\right]\\
&\le& \frac{C_K}{M} \sum_{i+j=r} \E \left[\sqrt{\int_K \left(\left|\partial_1^i \partial_2^j F_n\right|^2(x)+\left|\partial_1^{i+1} \partial_2^j F_n\right|^2(x)+\left|\partial_1^i \partial_2^{j+1} F_n\right|^2(x)\right)dx }\right]\\
&\le& \frac{C_K}{M} \sum_{i+j=r} \sqrt{\E \left[\int_K \left(\left|\partial_1^i \partial_2^j F_n\right|^2(x)+\left|\partial_1^{i+1} \partial_2^j F_n\right|^2(x)+\left|\partial_1^i \partial_2^{j+1} F_n\right|^2(x)\right)dx \right]}\\
&\le& \frac{r C_K \sqrt{3}}{M}.
\end{eqnarray*}
\end{proof}

\paragraph{Estimation of the small ball}
From Equation \eqref{laborne}, upper bounding the probability $\mathbb P(l_{n,k,l} >t)$ thus reduces to establish a small ball estimate for $F_n( k \pi, l\pi)$. 
In this paragraph, we shall indeed establish such a small ball estimate, for any $1<\theta<\frac 3 2$:
\begin{equation}\label{SBE}
\mathbb{P}\left(\left| F_n( k \pi, l\pi)\right| \le \epsilon \right)\le C \left( \epsilon+\frac{1}{n^\theta} \right),
\end{equation}
provided that $\text{ord}(k)\ge \sqrt{n}$ and $\text{ord}(l)\ge \sqrt{n}$. To proceed, we use the celebrated \textit{Halasz} method. First of all, for some absolute constant $C>0$, we have
\begin{eqnarray}\label{Halasz}
\mathbb{P}\left(\left| F_n( k \pi, l\pi)\right| \le \epsilon \right)\le C \epsilon \int_\mathbb{R} \Phi_{F_n(k\pi,l\pi)}(\xi) e^{-\frac{\epsilon^2\xi^2}{2}}d\xi,
\end{eqnarray}
where $\Phi_{F_n(k\pi,l\pi)}(\cdot)$ is the characteristic function of $F_n(k\pi,l\pi)$. Note that if $X$ is a random variable $X$ such that $\E(X)=0$, $\E(X^2)=1$, then we have $|\E(e^{i \xi X})|\le \exp(-\xi^2/4)$ on an interval $[-\alpha,\alpha]$ for $\alpha>0$ small enough. As a result we may first write
\begin{eqnarray*}
\int_{|\xi|\le \alpha n} \Phi_{F_n(k\pi,l\pi)}(\xi) e^{-\frac{\epsilon^2\xi^2}{2}}d\xi &\le& \int_{|\xi|\le \alpha n} \Phi_{F_n(k\pi,l\pi)}(\xi) d\xi\\
&\le&\int_{|\xi|\le \alpha n} \prod_{1\le i,j\le n} e^{-\frac{\xi^2}{4n^2}\cos^2(i\frac{k\pi}{n})\cos^2(j\frac{l\pi}{n})}d\xi
\end{eqnarray*}
However, based on the following doubling formula $\cos(2x)=2\cos(x)^2-1$, we have the following dichotomy: either $|\cos(x)|\ge \frac 1 2 $ either $|\cos(2x)|\ge \frac 1 2 $. We may then restrict our attention to the set of indexes $(i,j)$ such that $|\cos
(i\frac{k\pi}{n})|\ge \frac 1 2$ and $|\cos
(j\frac{l\pi}{n})|\ge \frac 1 2$ whose cardinality is hence necessarily larger than $\frac{n^2}{4}$. This entails that
\begin{equation}\label{bound-product1}
\prod_{1\le i,j\le n} e^{-\frac{\xi^2}{4n^2}\cos^2(i\frac{k\pi}{n})\cos^2(j\frac{l\pi}{n})}\le  \left(e^{-\frac{\xi^2}{64n^2}}\right)^{\frac{n^2}{4}}=e^{-\frac{\xi^2}{256}}.
\end{equation}
However, since $\xi\mapsto e^{-\frac{\xi^2}{256}} \in L^1(\mathbb{R})$, the bound (\ref{bound-product1}) implies the existence of an absolute constant $C>0$ such that
\begin{equation}\label{boundI1}
\sup_{n\ge 1,\epsilon>0} \int_{|\xi|\le \alpha n} \Phi_{F_n(k\pi,l\pi)}(\xi) e^{-\frac{\epsilon^2\xi^2}{2}}d\xi\le C
\end{equation}
As a result, bounding the right hand side of (\ref{Halasz}) requires the control of the integral
\begin{equation}\label{I2}
I_2:=\epsilon\int_{|\xi|\ge \alpha n} \Phi_{F_n(k\pi,l\pi)}(\xi) e^{-\frac{\epsilon^2\xi^2}{2}}d\xi.
\end{equation}
Now, relying on Lemma \ref{Lemma-computation}, we get
$$\Phi_{F_n(k\pi,l\pi)}(\xi)=\prod_{\substack{ 1\le i \le \text{ord}(k)\\ 1\le j \le \text{ord}(l)}}\Phi_a\left(\frac{\xi}{n}\cos\left(\frac{i\pi}{\text{ord}(k)}\right)\cos\left(\frac{j\pi}{\text{ord}(l)}\right)\right)^{\frac{n^2}{\text{ord}(k)\text{ord}(l)}},$$
where $\Phi_a$ naturally stands for the characteristic function of the common law of the coefficients. By writing $u=\frac \xi n$, the integral (\ref{I2}) becomes

\begin{equation}\label{expressionI2}
I_2:=n \epsilon \int_{|u|>\alpha}
\prod_{\substack{1\le i \le \text{ord}(k)\\ 1\le j \le \text{ord}(l)}}\Phi_a\left(u\cos\left(\frac{i\pi}{\text{ord}(k)}\right)\cos\left(\frac{j\pi}{\text{ord}(l)}\right)\right)^{\frac{n^2}{\text{ord}(k)\text{ord}(l)}} 
e^{-\frac{u^2\epsilon^2 n^2}{2}} du.
\end{equation}

Now, for fixed $A<B<1$ and $u\in \mathbb{R}/\{0\}$, we denote by $\phi=\phi_{A,B,u}:[-1,1]\to[0,1]$ the Lipschitz function such that $\phi(x)=1$ when $|\Phi_a(u x)|\le A$, $\phi(x)=0$ when $|\Phi_a(ux)|\ge B$ and $\phi$ is linear on $|\Phi_a(ux)|\in [A,B]$. Note that, for any $(x,y)\in \mathbb{R}^2$, if $\phi(x)=1$ and $\phi(y)=0$, then necessarily (since $\Phi_a$ is $1$-Lipschitz):

$$|ux-uy|\ge |\Phi_a(ux)-\Phi_a(uy)| \ge B-A.$$

Besides, if $|\Phi_a(uz)|\in ]A, B[$ one may always find an interval $(x,y)$ containing $z$ such that (i) $|\Phi_a(ux)|=A$ and $|\Phi_a(u y)|=B$, (ii) for all $w\in (x,y)$ it holds that $|\Phi_a(u w)| \in [A, B]$. Since by definition $\phi$ is linear on $(x,y)$ we may deduce that 

\begin{equation}\label{control if the derivative}
\left|\phi'(z)\right|=\left|\frac{\phi(x)-\phi(y)}{x-y}\right|\le \frac{|u|}{B-A}.
\end{equation}
As a result, by recognizing a two-dimensional Riemann sum of the bivariate function $\Psi(x,y):=\phi(\cos(\pi x)\cos(\pi y))$,
\begin{eqnarray}\label{boundLipschitz}
&&\nonumber\left|\frac{1}{\text{ord}(k)\text{ord}(l)}\sum_{i=1}^{\text{ord}(k)}\sum_{j=1}^{\text{ord}(l)}\phi\left(\cos\left(\frac{i\pi}{\text{ord}(k)}\right)\cos\left(\frac{j\pi}{\text{ord}(l)}\right)\right)-\int_{[0,1]^2}\Psi(x,y) dx dy \right|\\
&\le& \frac{\|\nabla \Psi\|_\infty}{\min(\text{ord}(k),\text{ord}(l))}\le C_{A,B}\frac {|u|}{\sqrt{n}}.
\end{eqnarray}
Note that, by construction, $\phi$ implicitely depends on $u,A,B$ but in order to alledge the notations we will note carry this dependency in our notations. Now denote by $\rho$ the density of the image measure of Lebesgue on $[0,1]^2$ by the functional $(x,y)\mapsto\cos(\pi x)\cos(\pi y)$ so that we have $\int_{[0,1]^2}\Psi(x,y) dx dy=\int_{\mathbb{R}}\phi(t)\rho(t)dt$. Since $\rho \in L^1(\mathbb{R})$, it is a well known fact that
$$\lim_{\delta\to 0}\sup_{\lambda(A)\le \delta} \int_A \rho(t)dt=0.$$
Let us fix $\delta_0>0$ such that $\sup_{\lambda(A)\le \delta_0} \int_A \rho(t)dt<\frac 1 2$. Nevertheless, one may fix $A,B>0$ (eventually close to $1$) such that $\sup_{|u|>\alpha}\lambda \left(\{\phi \neq 1\}\right)<\delta_0$. Let us detail a bit this assertion. First of all, we notice that
\begin{eqnarray*}
\lambda \left(\{\phi \neq 1\}\right)&=&\frac 1 u  \int_0^u \textbf{1}_{\{|\Phi_a(t)|>A\}} dt
=\frac 1 u  \int_0^u \textbf{1}_{\{|\Phi_a(t)|^2>A^2\}} dt\\
&\le& \frac{1}{u A^2} \int_0^u |\Phi_a(t)|^2 dt
= \frac{1}{A^2}\E\left(\sin_c\left(u (a_1-a_2)\right)\right).
\end{eqnarray*}
Assuming first that $a_1-a_2$ does not have an atom at zero, the dominated convergence theorem ensures that $\E\left(\sin_c\left(u (a_1-a_2)\right)\right)$ goes to zero as $u$ goes to infinity. Besides, for every fixed $u$, it holds that $\int_0^u \textbf{1}_{\{|\Phi_a(t)|<A\}} dt$ goes to zero as $A$ tends to one. Together, these two conditions ensure the desired result, namely
\[
\lim_{A\to 1} \sup_{|u|>c} \frac 1 u \int_0^u \textbf{1}_{\{|\Phi_a(t)|>A\}} dt=0.
\]
Assume now that $a_1-a_2$ has an atom at zero. Note that $a_1-a_2$ is not a constant variable since its variance is positive. Thus, for some $0<c<1$, one can write $\Phi_{a_1-a_2}=|\Phi_a|^2= c + (1-c) \Psi$ where $\Psi$ is the characteristic function of the law of $a_1-a_2$ conditionally to $a_1\neq a_2$. Since, $\textbf{1}_{\{|\Phi_a(t)|^2>A^2\}}\le \textbf{1}_{\{|\Psi|>\frac{A^2-c}{1-c}\}}$ (with $\frac{A^2-c}{1-c}\to 1$ as $A\to1$), we may apply the previous reasoning to the characteristic function $\Psi$ which by construction does not have an atom at zero.
Under these conditions we infer that
\begin{eqnarray*}
\int_{[0,1]^2}\Psi(x,y) dx dy=\int_{\mathbb{R}} \phi(t)\rho(t)dt
\ge \int_{\{\phi=1\}} \rho(t)dt
\ge \frac 1 2.
\end{eqnarray*}
And relying on the bound (\ref{boundLipschitz}), if one assumes that $|u|\le \frac{\sqrt{n}}{4 C_{A,B}}$, then we get the following crucial estimate
\begin{equation}\label{minoration-cardinal-SB}
\sum_{\substack{ 1\le i \le \text{ord}(k)\\ 1\le j \le \text{ord}(l)}}\phi\left(\cos\left(\frac{i\pi}{\text{ord}(k)}\right)\cos\left(\frac{j\pi}{\text{ord}(l)}\right)\right)\ge \frac{1}{4} \text{ord}(k) \text{ord}(l).
\end{equation}
which implies that the cardinality of couple of indexes $(i,j)$ such that 
\[
\left|\Phi_a\left(u\cos\left(\frac{i\pi}{\text{ord}(k)}\right)\cos\left(\frac{j\pi}{\text{ord}(l)}\right)\right)\right|\le B,
\]
is greater than $\frac{1}{4} \text{ord}(k) \text{ord}(l)$ provided that $\frac{\sqrt{n}}{4 C_{A,B}}>|u|>\alpha$. Coming back to (\ref{expressionI2}), we may infer that
\begin{eqnarray*}
I_2&\le&n \epsilon \int_{\frac{\sqrt{n}}{4 C_{A,B}}>|u|>\alpha}\Phi_{F_n(k\pi,l\pi)}(n u) e^{-\frac{u^2\epsilon^2 n^2}{2}} du\\
&+&n \epsilon \int_{\frac{\sqrt{n}}{4 C_{A,B}}<|u|}\Phi_{F_n(k\pi,l\pi)}(n u ) e^{-\frac{u^2\epsilon^2 n^2}{2}} du\\
&\le& \frac{B^{\frac{n^2}{4}}}{4 C_{A,B}} n \sqrt{n}+n \epsilon \int_{\frac{\sqrt{n}}{4 C_{A,B}}<|u|}e^{-\frac{u^2\epsilon^2 n^2}{2}} du\\
&=& \frac{B^{\frac{n^2}{4}}}{4 C_{A,B}} n \sqrt{n}+\int_{\frac{n\sqrt{n}\epsilon}{4 C_{A,B}}<|x|}e^{-\frac{x^2}{2}} dx.
\end{eqnarray*}
Now let us give the final argument of this proof. If $\epsilon \ge \frac{1}{n^\theta}$ then $n\sqrt{n}\epsilon\ge n^{\frac 3 2 -\theta}$ and $\int_{\frac{n\sqrt{n}\epsilon}{4 C_{A,B}}<|x|}e^{-\frac{x^2}{2}} dx=o\left(\frac {1}{n^\theta}\right)$. Otherwise, if $\epsilon< \frac{1}{n^\theta}$ then 
\begin{eqnarray*}
\mathbb{P}\left(\left| F_n( k \pi, l\pi)\right| \le \epsilon \right)&\le& \mathbb{P}\left(\left| F_n( k \pi, l\pi)\right| \le \frac{1}{n^\theta} \right)\\
&\le& C \left( \epsilon+\frac{1}{n^\theta} \right).
\end{eqnarray*}
\paragraph{Synthesis}
This paragraph makes the synthesis of the two previous subsections. Note that, in the sequel, $C$ stands for some universal constant which may change from line to line. Up to using Lemma \ref{bound-order} and doubling the size of the square on which we consider the nodal line, we will assume that $\text{ord}(k), \text{ord}(l), \text{ord}(k+1), \text{ord}(l+1) \ge \sqrt{n}$. As a matter of fact, relying on the main estimate (\ref{SBE}) and Lemma \ref{bound-Sobolev}, we get that
$$\mathbb{P}\left(\left| F_n( k \pi, l\pi)\right| \le M\frac{(\pi \sqrt{2})^{r-1}}{(r-1)!}\right)\le C\left(M\frac{(\pi \sqrt{2})^{r-1}}{(r-1)!}+\frac{1}{n^\theta}+\frac{C}{M}\right).$$
Making an optimization on $M$, we get
$$\mathbb{P}\left(\left| F_n( k \pi, l\pi)\right| \le M\frac{(\pi \sqrt{2})^{r-1}}{(r-1)!}\right)\le C\left(\sqrt{\frac{(\pi \sqrt{2})^{r-1}}{(r-1)!}}+\frac{1}{n^\theta}\right).$$
As a result, provided that $\theta>1+\alpha$ we get that the existence of an absolute constant $C>0$ such that
\begin{equation}\label{Moment-unif}
\sup_{n,l,k}\E(l_{n,k,l}^{1+\alpha})<C.
\end{equation}

\subsection{End of the proof} \label{lafin}

In this final subsection, we make the compilation of the content of all previous subsections to establish the global universality result stated in the introduction.
\begin{thm}\label{theo.final}
Whatever the law of the entries $(a_{k,l})_{k,l\ge 1}$, as $n$ tends to infinity, we have 
\[
\lim_{n \to +\infty} \frac{\mathbb E[ \ell_{[0,\pi]^2}(f_n) ]}{n} = \frac{\pi^2}{2 \sqrt{3}}.
\]
\end{thm}

\begin{proof}
Let us recall Equation \eqref{eq.sumlocal} which express the global expectation as the sum of the microscopic contributions 
\[
\frac{\mathbb E[ \ell_{[0,\pi]^2}(f_n) ]}{n} = \frac{1}{n^2}\sum_{0\leq k,l\le n-1}\E\left(l_{n,k,l}\right).
\]
Let us fix $\epsilon>0$, and let us introduce $I_\epsilon:=[\epsilon,1-\epsilon]$ and
\[
\mathcal{A}_{n,\epsilon}:=\left(n I_\epsilon\cap \mathbb{N}\right)^2.
\]
One first notice that $\# \left(\mathcal{A}_{n,\epsilon}\right) \approx (1-2\epsilon)^2 n^2$. Next, using the bound (\ref{Moment-unif}), we may infer that
\begin{equation}\label{Bound1}
\frac{1}{n^2} \sum_{(k,l) \in \mathcal{A}_{n,\epsilon}^c} \E(l_{n,k,l})\le \frac{C}{n^2} \#\left(\mathcal{A}_{n,\epsilon}^c\right)\le C\left(1-(1-2\epsilon)^2\right).
\end{equation}
Let us denote by $l_{\infty,k,l}$ the length of the nodal set of the limit Gaussian process $F_{\infty}$ in the square $[k\pi,(k+1)\pi] \times [l\pi,(l+1)\pi]$.
Now we shall prove that
\begin{equation}\label{Bound2}
\lim_{n \to +\infty}\sup_{(k,l)\in \mathcal{A}_{n,\epsilon}} \left|\E\left[l_{n,k,l}\right]-\E\left[l_{\infty,k,l}\right]\right|= 0.
\end{equation}
To do so, we denote by $(p_n,q_n)\in\mathcal{A}_{n,\epsilon}$ one couple of integers for which the above maximum is reached. Next, thanks to Proposition \ref{pro.unif} and Remark \ref{rem.nondege}, we infer that the process 
\[
G_n(\cdot,\cdot)=F_n( p_n \pi +\cdot, q_n\pi+\cdot)
\] 
converges to the non-degenerate stationary Gaussian process $G_\infty$. Besides, relying on Proposition \ref{cvgce-Finfini}, the same conclusion holds for the process $F_\infty(p_n \pi+\cdot,q_n\pi+\cdot)$. Hence, via the content of Subsection \ref{Locuniv}, we indeed obtain that
\begin{equation}\label{unif2}
\lim_{n \to +\infty} \E\left[\phi(l_{n,p_n,q_n})\right]-\E\left[\phi(l_{\infty,p_n,q_n})\right]=0,
\end{equation}
for any continuous bounded function. Finally, for any $M>0$, we have
\begin{eqnarray*}
&&\frac{1}{n^2} \sum_{0\leq k,l\le n-1} \E[l_{n,k,l}\textbf{1}_{\{l_{n,k,l}>M\}}]\\
&&\le \frac{C}{n^2}\sum_{0\leq k,l\le n-1} \mathbb{P}\left(l_{n,k,l}>M\right)^{\frac{\alpha}{1+\alpha}}\\
&&\le \frac{C'}{n^2}\sum_{0\leq k,l\le n-1}\frac{1}{M^{\frac{\alpha}{1+\alpha}}}=\frac{C'}{M^{\frac{\alpha}{1+\alpha}}}.\\
\end{eqnarray*}
As a result, using the limit (\ref{unif2}) and taking $M$ large enough, we indeed get the asymptotics \eqref{Bound2}.
Finally, putting (\ref{Bound1}) and (\ref{Bound2}) together with Theorem \ref{Casgaussien}, we get that
\begin{equation}\label{final}
\lim_{n \to +\infty}\left|\frac{1}{n^2}\sum_{k,l\le n}\E\left(l_{n,k,l}\right) -\frac{\pi^2}{2\sqrt{3}}\right| = 0,
\end{equation}
which is the desired result.

\end{proof}
\section*{Acknowledgement}
Guillaume Poly is grateful to the Vietnamese Institute of Advances Studies for funding one month research stint in Hanoi, which led to some part of this research.
\bibliographystyle{alpha}

\begin{thebibliography}{MPRW15}

\bibitem[ADL{\etalchar{+}}15]{azais2015local}
Jean-Marc Aza{\"\i}s, Federico Dalmao, Jos{\'e} Le{\'o}n, Ivan Nourdin, and
  Guillaume Poly.
\newblock Local universality of the number of zeros of random trigonometric
  polynomials with continuous coefficients.
\newblock {\em arXiv preprint arXiv:1512.05583}, 2015.

\bibitem[AP15]{angst2015universality}
J{\"u}rgen Angst and Guillaume Poly.
\newblock Universality of the mean number of real zeros of random trigonometric
  polynomials under a weak cramer condition.
\newblock {\em arXiv preprint arXiv:1511.08750}, 2015.

\bibitem[AW09]{azais}
Jean-Marc Aza{\"{\i}}s and Mario Wschebor.
\newblock {\em Level sets and extrema of random processes and fields}.
\newblock John Wiley \& Sons, Inc., Hoboken, NJ, 2009.

\bibitem[Fla16]{flasche2016expected}
Hendrik Flasche.
\newblock Expected number of real roots of random trigonometric polynomials.
\newblock {\em arXiv preprint arXiv:1601.01841}, 2016.

\bibitem[FLL15]{fyodorov2015number}
Yan~V Fyodorov, Antonio Lerario, and Erik Lundberg.
\newblock On the number of connected components of random algebraic
  hypersurfaces.
\newblock {\em Journal of Geometry and Physics}, 95:1--20, 2015.

\bibitem[GW12]{gayet2012betti}
Damien Gayet and Jean-Yves Welschinger.
\newblock Betti numbers of random real hypersurfaces and determinants of random
  symmetric matrices.
\newblock {\em arXiv preprint arXiv:1207.1579}, 2012.

\bibitem[IK16]{iksanov2016local}
Alexander Iksanov and Zakhar Kabluchko.
\newblock Local universality for real roots of random trigonometric
  polynomials.
\newblock {\em arXiv preprint arXiv:1601.05740}, 2016.

\bibitem[IM71]{ibragimov1971expected}
Il'dar~Abdullovich Ibragimov and NB~Maslova.
\newblock On the expected number of real zeros of random polynomials i.
  coefficients with zero means.
\newblock {\em Theory of Probability \& Its Applications}, 16(2):228--248,
  1971.

\bibitem[Let16a]{letendre2016expected}
Thomas Letendre.
\newblock Expected volume and euler characteristic of random submanifolds.
\newblock {\em Journal of Functional Analysis}, 270(8):3047--3110, 2016.

\bibitem[Let16b]{letendre2016variance}
Thomas Letendre.
\newblock Variance of the volume of random real algebraic submanifolds.
\newblock {\em arXiv preprint arXiv:1608.05658}, 2016.

\bibitem[MPRW15]{marinucci2015non}
Domenico Marinucci, Giovanni Peccati, Maurizia Rossi, and Igor Wigman.
\newblock Non-universality of nodal length distribution for arithmetic random
  waves.
\newblock {\em arXiv preprint arXiv:1508.00353}, 2015.

\bibitem[NS10]{nazarov2010random}
Fedor Nazarov and Mikhail Sodin.
\newblock Random complex zeroes and random nodal lines.
\newblock {\em arXiv preprint arXiv:1003.4237}, 2010.

\bibitem[ORW08]{oravecz2008leray}
Ferenc Oravecz, Ze{\'e}v Rudnick, and Igor Wigman.
\newblock The leray measure of nodal sets for random eigenfunctions on the
  torus.
\newblock In {\em Annales de l'institut Fourier}, volume~58, pages 299--335,
  2008.

\bibitem[RS01]{rusakov}
Alexander Rusakov and Oleg Seleznjev.
\newblock On weak convergence of functionals on smooth random functions.
\newblock {\em Math. Commun.}, 6(2):123--134, 2001.

\bibitem[RW08]{rudnick2008volume}
Ze{\'e}v Rudnick and Igor Wigman.
\newblock On the volume of nodal sets for eigenfunctions of the laplacian on
  the torus.
\newblock In {\em Annales Henri Poincare}, volume~9, pages 109--130. Springer,
  2008.

\bibitem[SZ99]{zelditch}
Bernard Shiffman and Steve Zelditch.
\newblock Distribution of zeros of random and quantum chaotic sections of
  positive line bundles.
\newblock {\em Comm. Math. Phys.}, 200(3):661--683, 1999.

\bibitem[TV14]{tao2014local}
Terence Tao and Van Vu.
\newblock Local universality of zeroes of random polynomials.
\newblock {\em International Mathematics Research Notices}, page rnu084, 2014.

\bibitem[Wig10]{wigman2010fluctuations}
Igor Wigman.
\newblock Fluctuations of the nodal length of random spherical harmonics.
\newblock {\em Communications in Mathematical Physics}, 298(3):787--831, 2010.

\bibitem[Wil91]{wilkins}
J.~Ernest Wilkins, Jr.
\newblock Mean number of real zeros of a random trigonometric polynomial.
\newblock {\em Proc. Amer. Math. Soc.}, 111(3):851--863, 1991.

\end{thebibliography}
\newcommand{\etalchar}[1]{$^{#1}$}

\end{document}